\documentclass[12pt]{amsart}
\usepackage{amssymb,latexsym}
\usepackage[english]{babel}

\makeatletter
\@namedef{subjclassname@2010}{%
  \textup{2010} Mathematics Subject Classification}
\makeatother

\newtheorem{thm}{Theorem}[section]
\newtheorem{cor}[thm]{Corollary}
\newtheorem{lem}[thm]{Lemma}

\begin{document}

\baselineskip=17pt

\title{Metrizable DH-spaces of the first category}

\author{Sergey Medvedev}
\address{South Ural State University, Chelyabinsk, 454080 Russia}
\email{medv@math.susu.ac.ru}

\begin{abstract}
We show that if a separable space $X$ has a meager open subset
containing a copy of the Cantor set $2^{\, \omega}$, then $X$ has
$\frak{c}$ types of countable dense subsets. We suggest a
generalization of the $\lambda$-set for non-separable spaces. Let
$X$ be an $h$-homogeneous $\Lambda$-set. Then $X$ is densely
homogeneous and $X \setminus A$ is homeomorphic to $X$ for every
$\sigma$-discrete subset $A \subset X$.
\end{abstract}



\maketitle

All spaces under discussion are metrizable.

We see the rapid growth of the theory of \textsf{CDH}-spaces
lately. For example, K. Kunen, A. Medini and L. Zdomskyy
\cite{KMZ} proved that if a separable metrizable space $X$ is not
\textsf{CB} but has a \textsf{CB} dense subset, then $X$ has
$\frak{c}$ types of countable dense subsets. They obtained a
similar result \cite[Theorem 16]{KMZ} for a non-Baire space with
the perfect set property for open sets.  In fact, it suffices (see
Theorem \ref{t3}) to have a meager open subset containing a copy
of the Cantor set $2^{\, \omega}$.

We introduce the $\Lambda$-sets as a generalization of the
$\lambda$-sets for non-separable spaces and consider its
properties. In particular, we improve (see Theorem \ref{t4}) the
result due to R. Hern\'{a}ndez-Guti\'{e}rrez, M.~Hru\v{s}\'{a}k
and J. van Mill \cite[Proposition 4.9]{HHrM} concerning
\textsf{CDH}-property of $h$-homogeneous $\lambda$-sets.

In the paper we are not dealing with the set-theoretic methods;
only topological methods are applied.

\section{{\textsf{DH}}-spaces}

For all undefined terms and notation see~\cite{Eng}.

$X \approx Y$ means that $X$ and $Y$ are homeomorphic spaces. A
separable topological space $X$ is \textit{countable dense
homogeneous} (briefly, \textsf{CDH}) if, given any two countable
dense subsets $A$ and $B$ of $X$, there is a homeomorphism  $h: X
 \rightarrow X$ such that $h(A)= B$. The \textit{type} of a countable
dense subset $D$ of a separable space $X$ is $\{h(D): h$ is a
homeomorphism of $X \}$. So a separable space is \textsf{CDH} if
and only if it has exactly one type of countable dense subsets.
Also notice that the maximum possible number of types of countable
dense subsets of a separable space is $\frak{c}$.

A space $X$ is \textit{densely homogeneous} (briefly, \textsf{DH})
provided that (1) $X$ has a $\sigma$-discrete dense subset and (2)
if $A$ and $B$ are two $\sigma$-discrete dense subsets of $X$,
then there is a homeomorphism  $h: X \rightarrow X$ such that
$f(A)= B$. One can check that no \textsf{DH}-space without
isolated points can be a $\sigma$-discrete space. Clearly, if $X$
is a separable metrizable space, then $X$ is \textsf{CDH}
$\Leftrightarrow$ $X$ is \textsf{DH}.

A space $X$ is called \textit{of the first category} (or
\textit{meager}) if it can be represented as a countable union of
nowhere dense sets. A space $X$ is \textit{completely Baire}
(briefly, \textsf{CB}) if every closed subspace of $X$ is a Baire
space.

Lemma \ref{le1} was obtained by the author \cite{med86}.
Independently it was proved for separable spaces by B. Fitzpatrick
Jr. and H-X. Zhou \cite{Fitz1}.

\begin{lem}\label{le1}
For a metric space $X$ the following are equivalent:

{\rm{1)}} the space $X$ is of the first category,

{\rm{2)}} $X$ contains a $\sigma$-discrete everywhere dense set
 of type $G_\delta$ without isolated points.
\end{lem}

Recall that a separable space in which every countable set is a
$G_\delta$-set is called a $\lambda$-\textit{set}. Likewise, a
space in which every $\sigma$-discrete set is a $G_\delta$-set is
called a $\Lambda$-\textit{set}. From Lemma \ref{le1} it follows
that every metrizable $\Lambda$-set without isolated points is of
the first category in itself.

The following statement is similar to \cite[Theorem 3.4]{Fitz1}.

\begin{thm}\label{t1}
Every \textsf{DH}, meager, metrizable space $X$ is a
$\Lambda$-set.
\end{thm}

\begin{proof}
Take a $\sigma$-discrete subset $A$ of $X$. By Lemma \ref{le1},
there is a $\sigma$-discrete dense $G_\delta$-set $B \subset X$.
Clearly, $A \cup B$ is a $\sigma$-discrete dense subset of $X$.
Then $A \cup B$ is a $G_\delta$-set in $X$ because $A \cup B =
h(B)$ for some homeomorphism  $h: X \rightarrow X$. Since $A$ is a
$G_\delta$-set in $A \cup B$, $A$ is a $G_\delta$-set in $X$.
\end{proof}

\begin{thm}\label{t2}
If a meager metrizable space $X$ contains a copy of the Cantor set
$2^\omega$, then $X$ is not \textsf{DH}.
\end{thm}

\begin{proof}
Let $2^\omega \approx F \subset X$. Clearly, $F$ is a closed
nowhere dense subset of $X$. Take a dense set $A \subset F$ such
that $A$ is homeomorphic to the rationals. By Lemma \ref{le1},
there is a $\sigma$-discrete dense $G_\delta$-set $B \subset X$.
The set $F \setminus (A \cup B)$ is not $F_\sigma$ in $F$ because
it is homeomorphic to the irrationals. Hence, $X \setminus (A \cup
B)$ is not $F_\sigma$ in $X$.  Since $A \cup B$ is not $G_\delta$
in $X$, there is no homeomorphism $f: X \rightarrow X$ with $f(B)=
A \cup B$. Thus, $X$ is not \textsf{DH}.
\end{proof}

\begin{thm}\label{t3}
Suppose a separable metrizable space $X$ has a meager open subset
containing a copy of the Cantor set $2^{\, \omega}$. Then $X$ has
$\frak{c}$ types of countable dense subsets.
\end{thm}

\begin{proof}
The set $V= \bigcup \{U: U$ is a meager open set in $X \}$ is the
largest open subset of $X$ which is meager in itself. One can
check that $Y = \{y \in X:$ every neighborhood of $y$ contains a
copy of the Cantor set $2^{\, \omega} \}$ is a closed subset of
$X$. Under the conditions of the theorem, $V \cap Y \neq
\emptyset$.

\textsf{Case 1}. Let $V \cap \mathrm{Int} ( Y) = \emptyset$. Then
$V \cap Y$ is a non-empty closed nowhere dense subset of $V$.
Choose a countable set $D \subset V$ such that $V \subseteq
\overline{D}$ and $D \cap Y = \emptyset$. If $X \setminus V \neq
\emptyset$, fix a countable set $D^*$ such that $\overline{D^*} =
X \setminus V$; otherwise put $D^*= \emptyset$. Take a copy $K
\subset V \cap Y$ of the Cantor set $2^{\, \omega}$ and consider a
family $\{E_\alpha \subset K: \alpha < \frak{c}\}$ of size
$\frak{c}$ consisting of pairwise non-homeomorphic countable
spaces; such family exists by the theorem due to S. Mazurkiewicz
and W. Sierpi{\'{n}}ski \cite{MzS}. For every $\alpha < \frak{c}$,
define $D_\alpha = E_\alpha \cup D \cup D^*$. Clearly, each
$D_\alpha$ is a countable dense subset of $X$.

We claim that $D_\alpha$ and $D_\beta$ are countable dense subsets
of $X$ of a different type whenever $\alpha \neq \beta$. Assume
that there exists a homeomorphism $h: X \rightarrow X$ such that
$h(D_\alpha) = D_\beta$. One readily sees that $h(V) = V$, $h(X
\setminus V) = X \setminus V$, and $h(V \cap Y) = V \cap Y$.
Hence, $E_\beta = D_\beta \cap V \cap Y = h(D_\alpha \cap V \cap
Y) = h(E_\alpha)$. This contradicts the fact that $E_\alpha$ and
$E_\beta$ are non-homeomorphic.

\textsf{Case 2}. Let $V^* = V \cap \mathrm{Int} ( Y)  \neq
\emptyset$. Now we will use ideas from the proof of \cite[Theorem
4.5]{HrM}. Fix a copy $K \subset V^*$ of the Cantor set $2^{\,
\omega}$. Clearly, $K$ is nowhere dense in $V^*$. Then there
exists an open set $W$ such that $\overline{W} \subset V^*$, $V^*
\setminus \overline{W} \neq \emptyset$, and the boundary
$\mathrm{Fr} (W) = K$. Obviously, $W \cap K = \emptyset$. Using
Lemma \ref{le1}, we can choose a countable dense subset $D \subset
W$ which is a $G_\delta$-set in $X$. Take a countable base $\{V_n:
n \in \omega \}$ for $V^* \setminus \overline{W}$. For every $n
\in \omega$ we may pick a closed nowhere dense set $F_n \subset
V_n$ such that $F_n \approx 2^{\, \omega}$. Without loss of
generality, $F_n \cap F_m = \emptyset$ whenever $n \neq m$. Let
$D_n$ be a countable dense subset of $F_n$. Clearly, each $D_n$ is
homeomorphic to the rationals. Put $D^\prime = \cup \{D_n: n \in
\omega \}$. Note that $D^\prime \cap O$ is not a $G_\delta$-subset
of $X$ for every non-empty open set $O \subseteq V^* \setminus
\overline{W}$. For if this were true, we could pick $n \in \omega$
such that $F_n \subset O$ which would imply that $D^\prime \cap
F_n = D_n$ would be $G_\delta$ in $F_n$. The last contradicts to
the Baire category theorem because the rationals cannot be
homeomorphic to a $G_\delta$-subset of a compact space.

Fix a countable dense subset $D^{\prime \prime}$ of $X \setminus
\overline{V^*}$. Put $D^*=D \cup D^\prime \cup D^{\prime \prime}$.

As above, there exists a family $\{E_\alpha \subset K: \alpha <
\frak{c}\}$ of size $\frak{c}$ consisting of pairwise
non-homeomorphic countable spaces. For every $\alpha < \frak{c}$,
define $D_\alpha = E_\alpha \cup D^*$. Clearly, each $D_\alpha$ is
a countable dense subset of $X$.

We claim that $D_\alpha$ and $D_\beta$ are countable dense subsets
of $X$ of a different type whenever $\alpha \neq \beta$. Assume
that there exists a homeomorphism $h: X \rightarrow X$ such that
$h(D_\alpha) = D_\beta$. One readily sees that $h(V) = V$ and
$h(Y) = Y$. Hence, $h(V^*) = V^*$. Moreover, let us show that
$h(W) \cap (V^* \setminus \overline{W}) = \emptyset$. Assume the
converse. Then there is a non-empty open set $O \subset W$ such
that $h(O) \subset V^* \setminus \overline{W}$. However, this is
impossible because $O \cap D$ is a $G_\delta$-set in $X$ while
$h(O \cap D) = h(O) \cap D^\prime $ is not. Likewise, we obtain $W
\cap h(V^* \setminus \overline{W}) = \emptyset$. Thus, $h(W) = W$
and $h(V^* \setminus \overline{W}) = V^* \setminus \overline{W}$.
Hence, $h(K)=K$ because $V^* = W \cup K \cup (V^* \setminus
\overline{W})$. Then $E_\beta = D_\beta \cap K = h(D_\alpha \cap
K) = h(E_\alpha)$. This contradicts the fact that $E_\alpha$ and
$E_\beta$ are non-homeomorphic.
\end{proof}

\section{$h$-Homogeneous {\textsf{DH}}-spaces}

A space $X$ is called \textit{$h$-homogeneous} if ${\mathrm{Ind}}
X =0$ and every non-empty clopen subset of $X$ is homeomorphic to
$X$. If $\mathcal{U}$ is a family of subsets of a metric space
$(X, \varrho)$, then $\cup \, \mathcal{U} = \cup \{U: U \in
\mathcal{U} \}$ and $\mathrm{mesh} (\mathcal{U})= \sup
\{\mathrm{diam} U: U \in \mathcal{U} \}$ is a \textit{measure} of
$\mathcal{U}$.

\begin{lem}\label{le-res}
Let $X_i$ be a metrizable space with ${\mathrm{Ind}} X_i =0$ and
$F_i$ be a nowhere dense closed subset of $X_i$, where $i \in \{1,
2 \}$. Let $f:F_1\rightarrow F_2$ be a homeomorphism.

Then there exist a $\sigma$-discrete (in $X_i$) cover
$\mathcal{V}_i$ of $X_i \setminus F_i$ by non-empty pairwise
disjoint clopen subsets of $X_i$, $i \in \{1, 2 \}$, and a
bijection $\psi: \mathcal{V}_1 \rightarrow \mathcal{V}_2$ such
that for any subsets $D_1 \subseteq \bigcup \mathcal{V}_1$ and
$D_2 \subseteq \bigcup \mathcal{V}_2$ and any bijection $g:D_1
\rightarrow D_2$ satisfying $g(D_1 \cap V) = D_2 \cap \psi(V)$ for
every $V \in \mathcal{V}_1$, the combination mapping $f
\triangledown g: F_1 \cup D_1 \rightarrow F_2 \cup D_2$ is
continuous at each point of $F_1$ and its inverse $(f
\triangledown g)^{-1}$ is continuous at each point of $F_2$.
\end{lem}

Under the conditions of Lemma \ref{le-res}, we say
\cite{Med_Closed} that:

1) the cover $\mathcal{V}_i$ \textit{forms a residual family} with
respect to $F_i$, where $i \in \{1, 2 \}$,

2) the bijection $\psi: \mathcal{V}_1 \rightarrow \mathcal{V}_2$
is \textit{agreed to} $f$.

\medskip

\textsc{Remark.} A construction similar to Lemma \ref{le-res} was
used by B.~Knaster and M.~Reichbach \cite{KnR} for separable
spaces and by A.~V.~Ostrovsky \cite{ost81} and the author
\cite{msu} for non-separable spaces. The last two papers were
written in Russian, therefore the proof of Lemma \ref{le-res} was
also given in \cite{Med_Closed}. Note that like statements are
obtained for arbitrary metrizable spaces in \cite{msu} and
\cite{Med_Closed}.

Independently Lemma \ref{le-res} was obtained by F. van Engelen
\cite{EngInf}. In his notation, the triple $\langle \mathcal{V}_1,
\mathcal{V}_2, \psi \rangle$ is called a \textit{Knaster-Reichbach
cover}, or \textit{KR-cover}, for $\langle X_1 \setminus F_1, X_2
\setminus F_2, f \rangle$. Sometimes this term is more convenient.

\medskip

\begin{thm}\label{t4}
Let $X$ be an $h$-homogeneous metrizable $\Lambda$-set. Then $X$
is \textsf{DH}.
\end{thm}

\begin{proof}
Let $w(X) = \tau$. Take two $\sigma$-discrete dense subsets $A$
and $B$ of $X$. From $h$-homogeneity of $X$ it follows that $A$
and $B$ are weight-homogeneous spaces of weight $\tau$. Then $A
\approx B \approx Q(\tau)$ by \cite[Theorem 1]{mq}, where
$Q(\tau)$ is the small $\sigma$-product of countably many discrete
spaces of cardinality $\tau$. Let $A = \cup \{A_n: n \in \omega
\}$, where each $A_n$ is discrete in $X$, and $B = \cup \{B_n: n
\in \omega \}$, where each $B_n$ is discrete in $X$.

Fix a metric $\varrho$ on $X$ which induces the original topology
on $X$. Since ${\mathrm{Ind}} X = \dim X = 0$, by \cite[Theorem
7.3.1]{Eng}, there exists a sequence $\mathcal{U}_{\, 0},
\mathcal{U}_{\, 1}, \ldots$ of discrete clopen covers of $X$ such
that $\mathrm{mesh} (\mathcal{U}_{\, n}) < 1/(n+1)$ and
$\mathcal{U}_{\, n+1}$ is a refinement of $\mathcal{U}_{\, n}$ for
each $n$. Obviously, the family $\mathcal{U} = \{U \in
\mathcal{U}_{\, n}: n \in \omega \}$ forms a base for $X$. For
each $U \in \mathcal{U}$ fix a homeomorphism $\varphi_U: X
\rightarrow U$. Then the sets $ A^* = A \cup \bigcup
\{\varphi_U(A): U \in \mathcal{U} \}$ and $ B^*= B \cup \bigcup
\{\varphi_U(B): U \in \mathcal{U}\}$ are $\sigma$-discrete in $X$.
By definition of a $\Lambda$-set, $A^*$ and $B^*$ are both
$G_\delta$-subsets of $X$. Since $X^*  = X \setminus (A^* \cup
B^*)$ is a dense subset of $X$, it is a meager subset of $X$. Then
$X^*$ can be represented as $\cup \{X_n: n \in \omega \}$, where
each $X_n$ is a nowhere dense closed subset of $X$. By
construction, $\varphi_U(X_n) \cap A = \emptyset$ and
$\varphi_U(X_n) \cap B = \emptyset$ for any $n \in \omega$ and $U
\in \mathcal{U}$.

The difference $X \setminus (A \cup X^*)$ is a $\sigma$-discrete
subset of $X$. Then $X \setminus (A \cup X^*) = \cup \{A^*_n: n
\in \omega \}$, where each non-empty $A^*_n$ is discrete in $X$.
Similarly, $X \setminus (B \cup X^*) = \cup \{B^*_n: n \in \omega
\}$, where each non-empty $B^*_n$ is discrete in $X$.

We will construct the homeomorphism $f: X \rightarrow X$
satisfying $f(A) = B$ by induction on $n$. For each $n \in \omega$
we will find a pair of closed nowhere dense sets $F_n$ and $E_n$,
a homeomorphism $f_n: F_n \rightarrow E_n$, a KR-cover $\langle
\mathcal{V}_n, \mathcal{W}_n, \psi_n \rangle$ for $\langle X
\setminus F_n, X \setminus E_n, f_n \rangle$ such that:

(a) $F_n \cup X_n \cup A^*_n \cup A_n \subset F_{n+1}$,

(b) $E_n \cup X_n \cup B^*_n \cup B_n \subset E_{n+1}$,

(c) $f_n(A \cap F_n) = B \cap E_n$,

(d) the restriction $f_n |_{F_i} = f_i$ for every $i < n$,

(e) $\mathrm{mesh} (\mathcal{V}_n) < 1/(n+1)$ and $\mathrm{mesh}
(\mathcal{W}_n) < 1/(n+1)$,

(f) $\mathcal{V}_{n+1}$ refines $\mathcal{V}_n$ and
$\mathcal{W}_{n+1}$ refines $\mathcal{W}_n$,

(g) if $V \in \mathcal{V}_n$ and $V^* \in \mathcal{V}_m$ for $n <
m$, then $V^* \subset V$ if and only if $\psi_m(V^*) \subset
\psi_n(V)$.

For $n=0$, let $F_0 =E_0= \emptyset$, $\mathcal{V}_0 =
\mathcal{W}_0 = \{X \}$, $\psi_0(X)= X$, and
$f_0(\emptyset)=\emptyset$.

Now let us consider the induction step $n+1$ for some $n \in
\omega$.

First, fix $V \in \mathcal{V}_n$ and $W = \psi_n(V) \in
\mathcal{W}_n$. Let $Y = V \cap X_m$, where $m$ is the least $i$
with $V \cap X_i \neq \emptyset$, and $Z = W \cap X_l$, where $l$
is the least $i$ with $W \cap X_i \neq \emptyset$.

Since $A$ is dense in $X$ and $X$ is nowhere locally compact and
weight-homogeneous, there exists a discrete (in $X$) set $A^\prime
\subset A \cap V$ of cardinality $\tau$ such that $A_{m} \cap V
\subset A^\prime$, where $m$ is the least $i$ with $V \cap A_i
\neq \emptyset$. We can also find a discrete (in $X$) set
$B^\prime \subset B \cap W$ of cardinality $\tau$ such that $B_{l}
\cap W \subset B^\prime$, where $l$ is the least $i$ with $W \cap
B_i \neq \emptyset$. Choose a bijection $g: A^\prime \rightarrow
B^\prime$.

Next, let $A^{* \prime} = V \cap A^*_m$, where $m$ is the least
$i$ with $V \cap A^*_i \neq \emptyset$, and $B^{* \prime} = W \cap
B^*_l$, where $l$ is the least $i$ with $W \cap B^*_i \neq
\emptyset$. Now we have to distinguish some possibilities.
\textsf{Case 1}: $A^{* \prime} \neq \emptyset$ and $B^{* \prime}
\neq \emptyset$. By adding points from $V \cap (X^* \setminus Y)$
or from $W \cap (X^* \setminus Z)$, if it is necessary, we may
assume that $A^{* \prime}$ and $B^{* \prime}$ are both discrete
subsets of $X$ of the same cardinality. \textsf{Case 2}: $A^{*
\prime} \neq \emptyset$ and $B^{* \prime} = \emptyset$. In this
case we replace $B^{* \prime}$ by a discrete subset from $W \cap
(X^* \setminus Z)$ of cardinality $|A^{* \prime}|$. \textsf{Case
3}: $A^{* \prime} = \emptyset$ and $B^{* \prime} \neq \emptyset$
is similar to \textsf{Case 2}. \textsf{Case 4}: $A^{* \prime} =
B^{* \prime} = \emptyset$. For \textsf{Cases 1--3} we can take a
bijection $h: A^{* \prime} \rightarrow B^{* \prime}$. For
\textsf{Case 4} we put $h(\emptyset)= \emptyset$.

The union $Y^{\prime} = Y \cup A^{\prime} \cup A^{* \prime}$ is a
closed nowhere dense subset of $V$. Hence there exists a basic set
$U \in \mathcal{U}$ such that $U \subset V$ and $U \cap Y^{\prime}
= \emptyset$. Then $\varphi_U(Z) \approx Z$ and $\varphi_U(Z)$ is
a closed nowhere dense subset of $V$. Define $F_V = Y^{\prime}
\cup \varphi_U(Z)$. Likewise, there exists $U^* \in \mathcal{U}$
such that $U^* \subset W$ and $\varphi_{U^*}(Y) \approx Y$ misses
the set $Z^{\, \prime} = Z \cup B^{\prime} \cup B^{* \prime}$. Put
$E_W = Z^{\, \prime} \cup \varphi_{U^*}(Y)$.

Then we define a homeomorphism $f_V: F_V \rightarrow E_W$ by the
rule
\[ f_V(x) =\left \{
\begin{array}{cl}
g(x) & \mbox{if  } x \in A^{\prime},  \\
h(x) & \mbox{if  } x \in A^{* \prime}, \\
\varphi_{U^*}(x) & \mbox{if  } x \in Y, \\
(\varphi_U)^{-1}(x) & \mbox{if  } x \in \varphi_U(Z).
\end{array}   \right. \]

By Lemma \ref{le-res}, there exists a KR-cover $\langle
\mathcal{V}_V, \mathcal{W}_V, \psi_V \rangle$ for $\langle V
\setminus F_V, W \setminus E_W, f_V \rangle$.

It is not hard to see that $F_{n+1}  = F_n \cup \bigcup \{F_V: V
\in \mathcal{V}_n \}$ and $E_{n+1} = E_n \cup \bigcup
\{E_{\psi_n(V)}: V \in \mathcal{V}_n \}$ are  closed nowhere dense
subsets of $X$. Clearly, the mapping
\[ f_{n+1}(x) =\left \{
\begin{array}{cl}
f_V(x) & \mbox{if  } x \in F_V \mbox{ for some } V
\in \mathcal{V}_n,  \\
f_n(x) & \mbox{if  } x \in F_n
\end{array}   \right. \]
is a homeomorphism between $F_{n+1}$ and $E_{n+1}$.

Define $\mathcal{V}_{n+1} = \{U \in \mathcal{V}_V: V \in
\mathcal{V}_{n} \}$ and $\mathcal{W}_{n+1} = \{\psi_V(U): U \in
\mathcal{V}_V, V \in \mathcal{V}_{n} \}$. Then $\psi_{n+1}=
\bigcup \{\psi_V: V \in \mathcal{V}_n \}$ is a bijection between
$\mathcal{V}_{n+1}$ and $\mathcal{W}_{n+1}$ which is agreed to
$f_{n+1}$. Thus, $\langle \mathcal{V}_{n+1}, \mathcal{W}_{n+1},
\psi_{n+1} \rangle$ is a KR-cover for $\langle X \setminus
F_{n+1}, X \setminus E_{n+1}, f_{n+1} \rangle$. If necessary, the
families $\mathcal{V}_{n+1}$ and $\mathcal{W}_{n+1}$ can be
refined to families with measure less than $1/(n+2)$. One can
check that all conditions (a)--(g) are satisfied. This completes
the induction step.

From (a) and (b) it follows that $X = \bigcup \{F_n: n \in \omega
\} = \bigcup \{E_n: n \in \omega \}$. The condition (d) implies
that the rule $f(x)= f_n(x)$ if $x \in F_n$ for some $n$ defines
the bijection $f: X \rightarrow X$. Moreover, $f$ is a
homeomorphism. Next, using (c), we obtain $f(A)= B$.
\end{proof}

\begin{cor}
Let $X \subset 2^{\, \omega}$ be an $h$-homogeneous $\lambda$-set.
Then $X$ is \textsf{CDH}.
\end{cor}

\begin{thm}\label{t5}
Let $X$ be an $h$-homogeneous metrizable $\Lambda$-set. If $A$ is
a $\sigma$-discrete subset of $X$, then $X$ is homeomorphic to $X
\setminus A$.
\end{thm}

\begin{proof}
Since $\dim X= 0$, there exists a sequence $\mathcal{U}_{\, 0},
\mathcal{U}_{\, 1}, \ldots$ of discrete clopen covers of $X$ such
that $\mathcal{U}_{\, n+1}$ is a refinement of $\mathcal{U}_{\,
n}$ for each $n$ and the family $\mathcal{U} = \{U \in
\mathcal{U}_{\, n}: n \in \omega \}$ forms a base for $X$. For
each $U \in \mathcal{U}$ fix a homeomorphism $\varphi_U: X
\rightarrow U$. Then the set $ A^* = A \cup \bigcup
\{\varphi_U(A): U \in \mathcal{U} \}$ is $\sigma$-discrete in $X$.
Hence, $ A^* = \cup \{A_n: n \in \omega \}$, where each $A_n$ is a
discrete closed subset of $X$. For every $U \in \mathcal{U}$ the
set $U \setminus A$ contains a closed copy of each $A_n$ because
the space $X$ is weight-homogeneous and nowhere locally compact.

By definition of a $\Lambda$-set, $A^*$ is a $G_\delta$-subset of
$X$. Since $X^* = X \setminus A^* $ is a dense subset of $X$, it
is meager in $X$. Then $X^* = \cup \{X_n: n \in \omega \}$, where
each $X_n$ is a nowhere dense closed subset of $X$. By
construction, $\varphi_U(X_n)$ is a closed nowhere dense subset of
$U$ and $\varphi_U(X_n) \subset U \setminus A$ for any $n \in
\omega$ and $U \in \mathcal{U}$.

Since $X= \bigcup \{A_n \cup X_n: n \in \omega \}$, by \cite[Lemma
4]{Med_Closed}, every non-empty relatively open subset of $X
\setminus A$ contains a closed (in $X \setminus A$) nowhere dense
copy of $X$. According to \cite[Theorem 3]{Med_Closed}, $X
\setminus A$ is homeomorphic to $h(X, k)$, where $k =w(X)$ and
$h(X, k)$ is the smallest $h$-homogeneous meager space containing
$X$. Clearly, $X = h(X, k)$.
\end{proof}

\begin{cor}
Let $X \subset 2^{\, \omega}$ be an $h$-homogeneous $\lambda$-set.
If $A$ is a countable subset of $X$, then $X$ is homeomorphic to
$X \setminus A$.
\end{cor}

\end{document}